\documentclass[]{article}
\usepackage{graphicx}
\usepackage{amsmath}
\usepackage{amssymb, mathrsfs}
\usepackage{amsthm}
\usepackage{pxfonts}
\usepackage{enumerate}
\usepackage{color}
\usepackage{mathdots}
\usepackage{sectsty}
\usepackage[hidelinks]{hyperref}
\usepackage{tikz}
\usepackage{caption}
\usepackage{adjustbox}
\usepackage{fancyhdr}
\usepackage{verbatim}
\usepackage{url}

\sectionfont{\scshape\centering\fontsize{11}{14}\selectfont}
\subsectionfont{\scshape\fontsize{11}{14}\selectfont}

\newcommand\shorttitle{Divisor function over values of quadratic polynomials}
\newcommand\authors{\small Kostadinka Lapkova and Nian Hong Zhou}

\hypersetup{
    colorlinks=true,       
    linkcolor=blue,          
    citecolor=cyan,        
}

\theoremstyle{plain}
\newtheorem{theorem}{Theorem}[section]
\newtheorem{lemma}[theorem]{Lemma}

\theoremstyle{remark}

\makeatletter

\newcommand{\Rmnum}[1]{\expandafter\@slowromancap\romannumeral #1@}

\def\rb{\mathbb R}

\def\zb{\mathbb Z}
\def\cb{{\mathbb C}}

\def\rint{\int\limits}

\def\res{\mathop{\rm{Res}}}
\def\bmf{{\mathfrak M}}
\def\smf{{\mathfrak m}}

\numberwithin{equation}{section}

\title{\large \bf ON THE AVERAGE SUM OF THE $K$-TH DIVISOR FUNCTION OVER VALUES OF QUADRATIC POLYNOMIALS}
\author{\small KOSTADINKA LAPKOVA ~and~ NIAN HONG ZHOU}

\date{}

\begin{document}

\maketitle




\begin{abstract}Let $F({\bf x})\in\mathbb{Z}[x_1,x_2,\dots,x_n]$ be a quadratic polynomial in $n\geq 3$ variables with a nonsingular quadratic part. Using the circle method we derive an asymptotic formula for the sum
$$
\Sigma_{k,F}(X; {\mathcal{B}})=\sum_{{\bf x}\in X\mathcal{B}\cap\mathbb{Z}^{n}}\tau_{k}\left(F({\bf x})\right),
$$
for $X$ tending to infinity, where $\mathcal{B}\subset\mathbb{R}^n$ is an $n$-dimensional box such that
$\min\limits_{{\bf x}\in X\mathcal{B}}F({\bf x})\ge 0$ for all sufficiently large $X$, and $\tau_{k}(\cdot)$ is the $k$-th divisor
function for any integer $k\ge 2$.
\end{abstract}

\tableofcontents
\section{Introduction}
The $k$-th divisor function is a generalisation of the divisor function $\tau(m)=\sum_{d|m}1$ which counts the number of ways $m$
can be written as a product of $k$ positive integer numbers. It is defined as
$$
\tau_{k}(m)=\#\{(x_1,x_2,...,x_{k})\in\zb_+^{k}: m=x_1x_2...x_{k}\},
$$
where we assume that $\tau_k(0)=0$.  For polynomials $F({\bf x})\in \zb[x_1,\ldots,x_n]$ consider the sums
$$T_k(F({\bf x}),X)=\sum_{\left|F({\bf x})\right|\le X}\tau_k(\left|F({\bf x})\right|)\,.$$ Understanding the average order of $\tau_k(m)$,
as it ranges over sparse sequences of values taken by polynomials, i.e. of $T_k(F,X)$, is a problem that has received a lot of attention.
\newline

The most studied case is naturally $k=2$. For $F({\bf x})=F(x_1,x_2)$ a binary irreducible cubic form Greaves \cite{MR0263761}
showed that there exist real constants $c_1>0$ and $c_2$ depending only on $F$, such that
$$
T_2(F({\bf x}),X)= c_1X^{2/3}\log X+c_2X^{2/3} +O_{\varepsilon, F}(X^{9/14+\varepsilon}),
$$
holds for any $\varepsilon>0$ as $X\rightarrow \infty$. If $F(x_1,x_2)$ is an irreducible quartic form, Daniel \cite{MR1670278}
proved that
$$
T_2(F({\bf x}),X)= c_1X^{1/2}\log X+O_{F}(X^{1/2}\log\log X),
$$
where $c_1>0$ is a constant depending only on $F$. It seems that $\deg F=4$ is the limit of the current available methods treating
divisor sums over binary forms. More related works on the cases $k=2$ and $n=2$ are e.g. la Bret{\`{e}}che and Browning \cite{MR2719554},
Browning \cite{MR2861076} and Yu \cite{MR1754029}. On the other hand, with their paper from 2012 Guo and
Zhai \cite{GuoZhai2012} revived the interest toward estimating asymptotically $T_2(F({\bf x}),X)$ for forms in $n\geq 3$ variables
using the classical circle method. After many other papers extending \cite{GuoZhai2012} and dealing with diagonal forms,
in a recent work Liu \cite{Liu2019} obtained an asymptotic formula for $T_2(F({\bf x}),X)$ for any nonsingular quadratic form $F$
in $n\geq 3$ variables.\newline

For the cases when $k\ge 3$ there are only few results in the literature. Friedlander and Iwaniec \cite{MR2289206} showed that
$$
\sum_{\substack{n_1^2+n_2^6\leq X\\ {\rm gcd}(n_1,n_2)=1}}\tau_3(n_1^2+n_2^6)=cX^{2/3}(\log X)^2
+O\left(X^{2/3}(\log X)^{7/4}(\log \log X)^{1/2}\right),
$$
where $c$ is a constant. Daniel \cite[(4.5)]{DanielThesis} described an asymptotic formula for $T_k(F({\bf x}),X)$ as $X\rightarrow\infty$ for any $k\geq 2$ for irreducible binary definite quadratic forms $F$ and in \cite[(4.7)]{DanielThesis} he proved an asymptotic formula for $T_3(F({\bf x}),X)$ as $X\rightarrow\infty$
for irreducible binary cubic forms $F$. Sun and Zhang \cite{SunZhang2016}, with the help of the circle method, obtained
\[
\sum_{1\le m_1, m_2, m_3\le X}\tau_3\left(m_1^2+m_2^2+m_3^2\right)=c_1X^3\log^2 X+c_2X^3\log X+c_3X^3+O_{\varepsilon}(X^{11/4+\varepsilon}),
\]
where $c_1, c_2, c_3$ are constants and $\varepsilon$ is any positive number. Finally Blomer \cite{Blomer} proved an asymptotic formula
for the sum $\Sigma_{k,F}(X; \mathcal{B})$ defined in \eqref{defSigma}, for any $k\geq 2$, where $F({\bf x})$ is a form of degree $k$ in $n=k-1$ variables, coming from incomplete norm form.
\newline

 In this paper we investigate the average sum of the $k$-th divisor function over values of quadratic polynomials $F(\bf x)$, not
 necessarily homogenous, in $n\ge 3$ variables for \emph{any} $k\ge 2$. Every $n$-variables quadratic polynomial can be written as
\begin{equation}\label{def:F}
F({\bf x})={\bf x}^TQ{\bf x}+{\bf L}^T{\bf x}+N\,
\end{equation}
where $Q\in\zb^{n\times n}$ is a symmetric matrix, ${\bf L}\in\zb^n$ and $N\in\zb$. Our only additional requirement is that $Q$ is nonsingular. Let $\mathcal{B}\subset\mathbb{R}^n$ be an $n$-dimensional box (i.e. a certain product of intervals)
such that $\min_{{\bf x}\in X\mathcal{B}}F({\bf x})\ge 0$ for all sufficiently large $X$,
and for each integer $k\ge 2$, consider the sum
\begin{equation}\label{defSigma}
\Sigma_{k,F}(X; \mathcal{B})=\sum_{{\bf x}\in X\mathcal{B}\cap\mathbb{Z}^{n}}\tau_{k}\left(F({\bf x})\right),
\end{equation}
as $X$ tends to infinity.
Let us also use the following notation for $q\in\zb_{+}$
\[\varrho_F(q)=\frac{1}{q^{n-1}}\#\{{\bf h}\pmod q: F({\bf h})\equiv 0\bmod q\}.\]

Our main result is the following.
\begin{theorem}\label{mth} Let $F({\bf x})$ and $\Sigma_{k,F}(X; \mathcal{B})$ be defined as in \eqref{def:F} and \eqref{defSigma},
respectively, where $Q$ is a nonsingular matrix. Then for any $\varepsilon>0$ there exist real constants $C_{k,0}(F)$, $C_{k,1}(F)$,..., and $C_{k,k-1}(F)$,
such that for $X$ tending to infinity we have the asymptotic formula
$$\Sigma_{k,F}(X; \mathcal{B})=\sum_{r=0}^{k-1}C_{k,r}(F)\int_{X\mathcal{B}}(\log F({\bf t}))^r\,d{\bf t}
+O\left(X^{n-\frac{n-2}{n+2}\min\left(1,\frac{4}{k+1}\right)+\varepsilon}\right),$$
where the implied constant depends on $F$, $k$, $\mathcal{B}$ and $\varepsilon$, and
$$C_{k,r}(F)=\frac{1}{r!}\sum_{t=0}^{k-r-1}\frac{1}{t!}
\left({\frac{\,d^tL(s;k,F)}{\,ds^t}}\bigg|_{s=1}\right) \res_{s=1}\left((s-1)^{r+t}\zeta(s)^k\right).$$
The function $L(s; k,F)$ has the Euler product presentation
$$L(s;k,F)=\prod_{p}\left(\sum_{\ell\ge 0}\frac{\varrho_F(p^{\ell})\left(\tau_{k}(p^{\ell})-p^{s-1}\tau_{k}(p^{\ell-1})\right)}{p^{\ell s}}\right)
\left(\frac{(1-p^{-s})^k}{1-p^{-1}}\right)\,,$$
with $\tau_{k}(x):=0$ for all $x\not\in\zb$, and it is absolutely convergent for all $\Re(s)>1/2$. In particular, the main term has a positive leading coefficient:
$$C_{k,k-1}(F)=\frac{1}{(k-1)!}
\prod_{p}\left(\sum_{\ell\ge 0}\frac{\varrho_F(p^{\ell})\tau_{k-1}(p^{\ell})}{p^{\ell }}\right)\left(1-\frac{1}{p}\right)^{k-1}>0.$$
\end{theorem}

First of all, we remark that since $F({\bf x})$ has a nonsingular quadratic part, the set of all zeros of $F({\bf x})=0$
has a Lebesgue measure $0$, so that the logarithm function in
the integrals in the terms is well defined.
 Note that we provide a formula with $k$ terms, where one can easily see that the main term is of magnitude
 $X^n(\log X)^{k-1}$ (when $r=k-1$) and the last secondary term is of magnitude $X^n$ (when $r=0$). Thus the error term is indeed of a smaller
 rate. \newline

Using Theorem \ref{mth}, one can get the asymptotic formula for $\Sigma_{2,F}(X,\mathcal{B})$ in the most studied case of $k=2$.
This recreates and extends the main Theorem of Liu \cite{Liu2019} also for non-homogenous quadratic polynomials, but also provides
different expressions for the coefficients. Naturally, they can be also computed explicitly for specific polynomials, a goal we have
 not pursued in the current paper. Theorem \ref{mth} also extends the formula \cite[(4.5)]{DanielThesis} of Daniel to quadratic polynomials in more than $2$ variables, further, it elucidates the form of the involved coefficients.

\paragraph{Notations.}
The symbols  $\mathbb{Z}_+$, $\mathbb{Z}$ and $\mathbb{R}$ denote the positive integers, the integers and the real numbers, respectively.
$e(z):=e^{2\pi i z}$, $\zeta(s)=\sum_{n\ge 1}n^{-s}$ is the Riemann zeta function, the letter $p$ always denotes a prime.
We make use of the $\varepsilon$-convention: whenever $\varepsilon$ appears in a statement, it is asserted that the statement is true for all real $\varepsilon $.
This allows us to write $x^{\varepsilon}\log x\ll x^{\varepsilon}$ and $x^{2\varepsilon}\ll x^{\varepsilon}$, for example. Furthermore,
if not specially specified, all the implied constants of this paper in $O$ and $\ll$ depend on $F$, $k$, $\mathcal{B}$ and $\varepsilon$.\newline

\section{The proof of Theorem \ref{mth}}
\subsection{Setting up the circle method}

The primary technique used in the proof of the main theorem is the circle method and more precisely its treatment by Pleasants \cite{Pleas}.
The recent work on quadratic forms in $n\ge 3$ variables of Liu \cite{Liu2019} uses the same circle method techniques,
i.e. Weyl differencing, that were already used for general quadratic multivariable polynomials by Pleasants.

\par For the real $X$ from the definition \eqref{defSigma} let $L\ll X$ be a positive real parameter which we will choose later in a suitable way,
let $a,q\in\zb$, $0\le a<q\le L$ and $ {\rm gcd}(a,q)=1$. Then we define the intervals
$$\bmf_{a,q}(L):=\left\{\alpha\in\left[0,1\right]: \left|\alpha - \frac{a}{q}\right|\le \frac{L}{qX^2}\right\}.$$
The set of the major arcs is then the union
\begin{equation}\label{defMajor}
\bmf(L)=\bigsqcup_{\substack{0\le a<q\le L\\ {\rm gcd}(a,q)=1}}\bmf_{a,q}(L),
\end{equation}
and the set of the minor arcs is the complement $\smf(L)=\left[0,1\right]\setminus \bmf(L)$.

\par We further define the following exponential sums for $\alpha\in\rb$
$$S(\alpha)=\sum_{{\bf x}\in X\mathcal{B}\cap\zb^n}e\left(F({\bf x})\alpha\right)$$
and
$$T(\alpha,Y)=\sum_{0\le m\le Y}\tau_k(m)e(m\alpha). $$

Then, by the well-known identity for $u\in\zb$
\[\int_0^1 e(u\alpha)\,d\alpha=\left\{\begin{array}{lll}
1 & \text{if} &u=0,\\
0 & \text{if} &u\neq 0,
\end{array}\right.\]

we have
\begin{align*}
\Sigma_{k,F}(X; \mathcal{B})&=\sum_{{\bf x}\in X\mathcal{B}\cap\zb^{n}}\tau_{k}\left(F({\bf x})\right)\\
&=\int_{0}^{1}S(\alpha)T(-\alpha, C_{F,\mathcal{ B}}(X))\,d\alpha\\
&=\int_{\bmf(L)}S(\alpha)T(-\alpha, C_{F,\mathcal{ B}}(X))\,d\alpha+\int_{\smf(L)}S(\alpha)T(-\alpha, C_{F,\mathcal{ B}}(X))\,d\alpha\\
&=I_{\bmf(L)}+I_{\smf(L)},
\end{align*}
where
$$C_{F,\mathcal{ B}}(X):=\max_{{\bf x}\in X\mathcal{B}}|F({\bf x})|=X^2\max_{{\bf x}\in \mathcal{B}}\left|{\bf x}^TQ{\bf x}\right|+O(X)\asymp X^2.$$
We shall prove in Section \ref{ssmin} that for the contribution from the minor arcs we have
\begin{equation}\label{estm1}
I_{\smf(L)}\ll X^{n+\varepsilon}L^{1-n/2},
\end{equation}
as long as $L\ll X$. Already here we see that we need to require that the number of variables satisfy $n\geq 3$ in order to have an
error term of a smaller magnitude than $O(X^n)$.
Further, in Section \ref{ssmar} we will show that
\begin{equation}\label{estm2}
I_{\bmf(L)}-\sum_{r=0}^{k-1}C_{k,r}(F)\int\limits_{X\mathcal{B}}\left(\log(F({\bf t}))\right)^r\,d {\bf t}
\ll X^{n+\varepsilon}\left( L^{1-n/2}+L^2X^{-\min\left(1,\frac{4}{k+1}\right)}\right).
\end{equation}
Here for $r=0,1,\dots,k-1$,
\begin{equation}\label{def:Ckr}
C_{k,r}(F)=\sum_{q=1}^{\infty}\beta_{k,r}(q)S_F(q)\,,
\end{equation}
where
\begin{equation}\label{def:S_F}
S_F(q)=\sum_{\substack{a\in[1,q]\cap\zb\\ \gcd(a,q)=1}}q^{-n}\sum_{{\bf h}\in [1,q]^n\cap\zb^{n}}e\left(\frac{a}{q}F({\bf h})\right),
\end{equation}
$$\beta_{k,r}(q)=\frac{1}{r!}\sum_{t=0}^{k-r-1}\frac{1}{t!}\res_{s=1}\left((s-1)^{r+t}\zeta(s)^k\right)\left(\frac{\,d^t \Phi_k(q,s)}{\,ds^t}\bigg|_{s=1}\right),$$
and the analytic function $\Phi_k(q,s)$ is defined by Lemma \ref{t22}. We further consider the function
\begin{equation}\label{def:L}
L(s; k,F)=\sum_{q\ge 1}\Phi_k(q,s)S_F(q)\,.
\end{equation}
In Subsection \ref{lemss} we prove that it satisfies
\begin{equation}\label{eqFL}
L(s;k,F)=\prod_{p}\left(\sum_{\ell\ge 0}\frac{\varrho_F(p^{\ell})\left(\tau_{k}(p^{\ell})-p^{s-1}\tau_{k}(p^{\ell-1})\right)}{p^{\ell s}}\right)\left(\frac{(1-p^{-s})^k}{1-p^{-1}}\right)\,,
\end{equation}
with $\tau_{k}(x):=0$ for all $x\not\in\zb$.

Then Theorem \ref{mth} follows from \eqref{estm1},  \eqref{estm2} and \eqref{eqFL}, after choosing $L=X^{\frac{2}{n+2}\min\left(1, \frac{4}{k+1}\right)}$.


\subsection{Contribution from the minor arcs}\label{ssmin}~
Clearly, if the positive real numbers $L$ and $L'$ satisfy $L\le L'$, then $\bmf(L)\subset \bmf(L')$, and if $L\ge X$, then $[0, 1]\subset  \bmf(L)$
follows from Dirichlet's approximation theorem.

We further define
$$\mathscr{F}(L)=\bmf(2L)\setminus \bmf(L).$$
Then for a given positive number $L<X/2$,
$$[0, 1]\subset\bmf(L)\sqcup\bigsqcup_{0\le j<N}  \mathscr{F}(2^jL),$$
where $N$ is the smallest integer greater than or equal to $ (\log (X/L))/\log 2$.
 Clearly, the set of the small arcs then satisfy
\begin{equation}\label{mind}
\smf(L)\subset \bigsqcup_{0\le j<N}  \mathscr{F}(2^jL).
\end{equation}
To prove the estimate \eqref{estm1} over the minor arcs, we would use separate estimates of the two components $S(\alpha)$ and $T(\alpha,X)$
when $\alpha\in\mathscr{F}(L)$. We first state the following result.
\begin{lemma}\label{lemqe}For all positive numbers $L\ll X$,
$$\sup_{\alpha\in \mathscr{F}(L)}|S(\alpha)|\ll X^{n+\varepsilon}L^{-n/2}.$$
\end{lemma}
\begin{proof}
This estimate was done by Pleasants \cite{Pleas} even for the range $L\ll X(\log X)^{1/4}$. In the first equation of p.138 \cite{Pleas}
he proves that for $\alpha\in \mathscr{F}(L)$ we have $$|S(\alpha)|\le X^n(\log X)^n L^{-r/2},$$ where $r\ge 3$ is the rank of $Q$,
and in our case we have assumed that $r=n$.
\end{proof}
We also need the following estimate.
\begin{lemma}\label{lemde} For all positive numbers $L\ll X$,
$$\int_{\mathscr{F}(L)}\left|T(-\alpha, C_{F,\mathcal{ B}}(X))\right|\,d\alpha\ll X^{\varepsilon}L.$$
\end{lemma}
\begin{proof}By Cauchy's inequality, and using the definition of the major arcs \eqref{defMajor}, we have
\begin{align*}
\int_{\mathscr{F}(L)}\left|T(-\alpha, C_{F,\mathcal{ B}}(X))\right|\,d\alpha &\ll |\mathscr{F}(L)|^{1/2}
\left(\int_{0}^1\left|T(-\alpha, C_{F,\mathcal{ B}}(X))\right|^2\,d\alpha\right)^{1/2}\\
&\ll |\bmf(2L)|^{1/2}\left(\left(\sum_{1\le n\le C_{F,\mathcal{ B}}(X)}\tau_k(n)\right)^2\right)^{1/2}\\
&\ll \left(\sum_{1\le q\le L}\frac{2L}{qX^2}\varphi(q)\right)^{1/2}X^{1+\varepsilon}\ll X^{-1+1+\varepsilon}L\ll X^{\varepsilon}L,
\end{align*}
where we also applied the well known bound $\tau_k(n)\ll_k n^{\varepsilon}$ and the trivial $\varphi(q)/q\leq 1$.
\end{proof}
Now the estimate \eqref{estm1} over the minor arcs follow from \eqref{mind}, Lemma \ref{lemqe} and Lemma \ref{lemde}, namely
\begin{align*}
I_{\smf(L)}\ll& \sum_{0\le j<N}\int_{\mathscr{F}(2^jL)}\left|S(\alpha)T(-\alpha, C_{F,\mathcal{B}}(X))\right|\,d\alpha\\
\ll&\sum_{0\le j<N}\sup_{\alpha\in \mathscr{F}(2^jL)}|S(\alpha)|\int_{\mathscr{F}(2^jL)}\left|T(-\alpha, C_{F,\mathcal{B}}(X))\right|\,d\alpha\\
\ll &\sum_{0\le j<N} X^{n+\varepsilon}(2^jL)^{-n/2} (X^{\varepsilon}2^jL)\ll X^{n+\varepsilon}L^{1-n/2},
\end{align*}
where we used that $N\ll \log{X}$.

\subsection{Contribution from the major arcs}\label{ssmar}~

In this subsection we have $\alpha\in\bmf_{a,q}(L)$, and we shall write $\beta=\alpha-a/q$ for the coprime integers $a$ and $q$, $|\beta|\le L/(qX^2)$
and $1\le q\le L$. 
In order to prove the asymptotic formula \eqref{estm2}, we need the following statements.

\begin{lemma}\label{lemqqs} For $\alpha\in\bmf_{a,q}(L)$, and $\beta=\alpha-a/q$, we have
$$
S(\alpha)=q^{-n}S_F(q, a)\int_{X\mathcal{B}} e\left(F({\bf t})\beta\right)\,d{\bf t}+O_{\mathcal{B}, F}\left(LX^{n-1}\right),
$$
where
$$S_F(q,a)=\sum_{{\bf h}\in [1,q]^n\cap\zb^{n}}e\left(\frac{a}{q}F({\bf h})\right).$$
\end{lemma}
\begin{proof}
To prove this result we only need to adjust the last equation in the proof of \cite[Lemma 8]{Pleas} with the upper bounds $\beta\le L/(qX^2)$ and $q\le L$.
Note that Pleasant does the analysis over a quadratic polynomial with linear coefficients which can depend on $X$.
We are dealing with a quadratic $F$ with fixed coefficients, which makes the proof even easier.
\end{proof}

\begin{lemma}\label{lemqf} Let $S_F(q,a)$ be defined as in Lemma \ref{lemqqs}.  We have
$$
S_F(q, a)\ll_F q^{n/{2}+\varepsilon},
$$
where the implied constant is independent of $a$ and $q$
\end{lemma}
\begin{proof}
This is \cite[Lemma 10]{Pleas}.

\end{proof}
We further need the following two statements. The first one gives a general asymptotic representation of $T(\alpha, Y)$ and the second
one estimates the part of the singular integral coming from the major arcs. The proofs of Lemma \ref{lemkdf} and Lemma \ref{lemsi} will be given in
Section \ref{sec32} and Section \ref{sec41}, respectively.
\begin{lemma}\label{lemkdf}Let $ Y\asymp X^2$. We have
$$
T(\alpha, Y)=\sum_{r=0}^{k-1}\beta_{k,r}(q)\int_{0}^Y(\log u)^re(u\beta)\,du+O_{k,\varepsilon}\left(LX^{2-\frac{4}{k+1}+\varepsilon}\right).
$$
where for $r=0,1,\dots,k-1$,
$$\beta_{k,r}(q)=\frac{1}{r!}\sum_{t=0}^{k-r-1}\frac{1}{t!}\res_{s=1}\left((s-1)^{r+t}\zeta(s)^k\right)
\left(\frac{\,d^t \Phi_k(q,s)}{\,ds^t}\bigg|_{s=1}\right)\ll q^{-1+\varepsilon}.$$
with $\Phi_k(q,s)$ defined by Lemma \ref{t22}. In particular,
 $$
T(\alpha, Y)\ll q^{-1+\varepsilon} X^{2+\varepsilon}.
$$
\end{lemma}

\begin{lemma}\label{lemsi}We have
\begin{align*}
\int\limits_{\left|\beta\right|\le L/{qX^2}}\,d\beta\int\limits_{X\mathcal{B}}\,d{\bf t}
&\int\limits_{0}^{C_{F,\mathcal{B}}(X)}e\left((F({\bf t})-u)\beta\right)(\log u)^r\,du = \int\limits_{X\mathcal{B}}(\log F({\bf t}))^r\,d{\bf t}
+O\left(\frac{q^{n/2}X^{n+\varepsilon}}{L^{n/2}}\right).
\end{align*}
\end{lemma}

We now prove the asymptotic formula \eqref{estm2}. Using \eqref{defMajor} we get
\begin{align*}
I_{\bmf(L)}&=\int_{\bmf(L)}S(\alpha)T(-\alpha, C_{F,\mathcal{B}}(X))\,d\alpha\\
&=\sum_{q\le L}\sum_{\substack{0\le a< q\\ {\rm gcd}(a,q)=1}}\int_{|\beta|\le L/qX^2}S(\alpha)T(-\alpha, C_{F,\mathcal{B}}(X))\,d\alpha\\
&=:\sum_{q\le L}\sum_{\substack{0\le a< q\\ {\rm gcd}(a,q)=1}}\mathcal{I}_{q,a}.
\end{align*}
Since $1\le q\le L\ll X$, we have
\begin{align*}
\mathcal{I}_{q,a}
=&\rint_{|\beta|\le L/qX^2}S(\alpha)T(-\alpha, C_{F,\mathcal{B}}(X))\,d\beta\\
=&\rint_{|\beta|\le L/qX^2}\left(\frac{S_F(q, a)}{q^{n}}\rint_{X\mathcal{B}} e\left(F({\bf t})\beta\right)\,d{\bf t}\right)
T(-\alpha, C_{F,\mathcal{B}}(X))\,d\beta\\
&+O\left(\int_{|\beta|\le L/qX^2}\left(LX^{n-1}\right)q^{-1+\varepsilon}X^{2+\varepsilon}\,d\beta\right),\\
\end{align*}
by Lemma \ref{lemqqs}. Further, by applying Lemma \ref{lemqf}, Lemma \ref{lemkdf} and Lemma \ref{lemsi}, we get
\begin{align*}
\mathcal{I}_{q,a}
=&\sum_{r=0}^{k-1}\frac{S_F(q,a)\beta_{k,r}(q)}{q^n}
\int\limits_{\left|\beta\right|\le L/qX^2}\,d\beta\int\limits_{X\mathcal{B}} e\left(F({\bf t})\beta\right)\,d{\bf t}
\int_{0}^{C_{F,\mathcal{B}}(X)}(\log u)^re(-u\beta)\,d u\\
&+O\left(\int_{|\beta|\le L/qX^2}q^{-n/2+\varepsilon}X^n\left(LX^{2-\frac{4}{k+1}+\varepsilon}\right)\,d\beta+\frac{L^2X^{n-1+\varepsilon}}{q^2}\right)\\
=&\sum_{r=0}^{k-1}\frac{S_F(q,a)\beta_{k,r}(q)}{q^n}\int\limits_{X\mathcal{B}}(\log F({\bf t}))^r\,d{\bf t}+O\left(\frac{X^{n+\varepsilon}}{qL^{n/2}}
+\frac{X^{n-\frac{4}{k+1}+\varepsilon}L^2}{q^{1+n/2}}+\frac{L^2X^{n-1+\varepsilon}}{q^2}\right).
\end{align*}

Recall the notation \eqref{def:Ckr} and note that \[S_F(q)=\sum_{\substack{a\in[1,q]\cap\zb\\ \gcd(a,q)=1}}q^{-n}S_F(q,a).\]

Then after summing over all $1\le  q\le L$ and $1\leq a<q,\, \gcd(a,q)=1,$ the major arcs ${\bmf(L)}$ contribute
\begin{align*}
I_{\bmf(L)}=&\sum_{q\ge 1}S_F(q)\sum_{r=0}^{k-1}\beta_{k,r}(q)\int\limits_{X\mathcal{B}}(\log F({\bf t}))^r\,d {\bf t}
+O\left(\sum_{q>L}q^{-1+\varepsilon}|S_F(q)|X^{n+\varepsilon}\right)\\
&+O\left(X^{n+\varepsilon}L^{1-n/2}+X^{n-\frac{4}{k+1}+\varepsilon}L^2+L^2X^{n-1+\varepsilon}\right)\\
=&\sum_{r=0}^{k-1}C_{k,r}(F)\int_{X\mathcal{B}}(\log(F({\bf t})))^r\,d {\bf t}+O\left(X^{n+\varepsilon}E\right),
\end{align*}
with
\begin{align*}
E=L^{1-n/2}+L^2\left(X^{-1}+X^{-\frac{4}{k+1}}\right)\ll L^{1-n/2}+L^2X^{-\min\left(\frac{4}{k+1}, 1\right)}.
\end{align*}
Note that at this step, and at few other places, in order to control the error terms we necessarily have $n\geq 3$. This completes the proof of \eqref{estm2}.

\section{The estimates involving the $k$-th divisor function}\label{sec32}~

 The usual technique in estimating asymptotically through the circle method average sums similar to $\Sigma_{k,F}(X,\mathcal{ B})$,
 is the application of
 non-trivial average estimates of the specific arithmetic function over arithmetic progressions ( e.g. \cite{GuoZhai2012}, \cite{HuYang2018}, \cite{Liu2019}).
 Thus in order to prove Lemma \ref{lemkdf} we first need the following result.
\begin{lemma}\label{t21}Let $h,q$ be integers such that $1\leq h\le q$ and ${\rm gcd}(h,q)=\delta$. Then for each real number $x>1$, $q\le x^{\frac{2}{k+1}}$
and $\varepsilon>0$, we have
$$
A_k(x;h,q):=\sum_{\substack{m\le x\\ m\equiv h~(\bmod q)}}\tau_{k}(m)=M_{k}(x;h,q)+O_{k,\varepsilon}(x^{1-\frac{2}{k+1}+\varepsilon}),
$$
where
$$
M_{k}(x;h,q)=\res_{s=1}\left(\zeta(s)^{k}\frac{x^{s}}{s}f_{k}(q, \delta,s)\right)
$$
with
\begin{equation*}
f_{k}(q,\delta,s)=\frac{1}{\varphi(q/\delta)\delta^s}\left(\sum_{d |(q/\delta)}\frac{\mu(d)}{d^s}\right)^{k}
\sum_{d_1d_2...d_{k}=\delta}\sum_{\substack{t_i|(\prod_{j=i+1}^kd_{j})\\ \gcd(t_i,q/\delta)=1\\ i=1,2,...,k}}\frac{\mu(t_1)\dots \mu(t_k)}{\left(t_1...t_{k}\right)^{s}},
\end{equation*}
where $d_1, d_2,\ldots, d_{k}$ are positive integers and the empty product $\prod_{j=k+1}^kd_k:=1$.
\end{lemma}
\begin{proof} This lemma is essentially due to Smith \cite{Smith}, and we only adjust it for our purposes. We will extend easily \cite[Theorem 3]{Smith},
which covers the case when $h$ and $q$ are coprime, to any $h$ and $q$. First, equation (30) of \cite{Smith} states that
$$
A_k(x;h,q)=\sum_{d_1d_2...d_{k}=\delta}\sum_{\substack{t_i|(\prod_{j=i+1}^kd_{j})\\ i=1,2,...,k\\{\rm gcd}(t_1t_2...t_{k},q/\delta)=1}}\mu({\bf t})
A_k\left(\frac{x}{\delta t_1t_2...t_{k}};\overline{t_1t_2...t_k}\frac{h}{\delta},\frac{q}{\delta}\right),$$
where $d_1,d_2,\ldots,d_r$ are positive integers, $\mu({\bf t})=\prod_{j=1}^k\mu(t_j)$ and $\overline{m}$ is the multiplicative inverse of $m$ modulo $q$.
Then Theorem 3 of \cite{Smith} states that
$$A_k(x;h,q)=M_{k}(x;h,q)+\Delta_k(x;h,q),$$
where
$$
M_{k}(x;h,q)=\sum_{d_1d_2...d_{k}=\delta}\sum_{\substack{t_i|(\prod_{j=i+1}^kd_{j})\\ \gcd(t_i,q/\delta)=1\\ i=1,2,...,k}}\mu({\bf t})
\frac{x}{\delta t_1t_2...t_{k}}
P_{k}\left(\log\left(\frac{x}{\delta t_1t_2...t_{k}}\right),\frac{q}{\delta}\right)
$$
and
\begin{align*}
\Delta_k(x;h,q)=&\sum_{d_1d_2...d_{k}=\delta}\sum_{\substack{t_i|(\prod_{j=i+1}^kd_{j})\\ \gcd(t_i,q/\delta)=1\\ i=1,2,...,k}}
\mu({\bf t})\left(D_k\left(0;\overline{t_1...t_k}\frac{h}{\delta},\frac{q}{\delta}\right)\right)\\
&+\sum_{d_1...d_{k}=\delta}\sum_{\substack{t_i|(\prod_{j=i+1}^kd_{j})\\ \gcd(t_i,q/\delta)=1\\ i=1,2,...,k}}\mu({\bf t})
\left(O\left(\left(\frac{x}{\delta t_1...t_{k}}\right)^{\frac{k-1}{k+1}}\tau_{k}\left(\frac{q}{\delta}\right)\log^{k-1} (2x)\right)\right).
\end{align*}
Here $P_k(\log x, q)$ is a polynomial in $\log x$ of degree $k-1$ and $D_k(s;h,q)$ is the Dirichlet series corresponding to the
sum $A_k(x;h,q)$. By the definition of $P_k(\log x, q)$, namely \cite[(13)]{Smith},
and the analysis of $D_k(s;h,q)$ given in particular in \cite[(21)]{Smith}, it is easily seen that
\[
xP_k(\log x, q)=\frac{1}{\varphi(q)}\res_{s=1}\left(\left(\zeta(s)\sum_{d |q}d^{-s}\mu(d)\right)^{k}\frac{x^{s}}{s}\right).
\]
Hence
\begin{align*}
M_{k}(x;h,q)&=\sum_{d_1...d_{k}=\delta}\sum_{\substack{t_i|(\prod_{j=i+1}^kd_{j})\\ \gcd(t_i,q/\delta)=1\\ i=1,2,...,k}}\frac{\mu(t_1)\dots \mu(t_k)}{\varphi(q/\delta)}\res_{s=1}\left(\left(\zeta(s)\sum_{d |(q/\delta)}\frac{\mu(d)}{d^s}\right)^{k}\frac{x^s/s}{\left(\delta t_1...t_{k}\right)^{s}}\right)\\
&=\res_{s=1}\left(\frac{\zeta(s)^{k}x^s/s}{\varphi(q/\delta)\delta^s }\left(\sum_{d |(q/\delta)}\frac{\mu(d)}{d^s}\right)^{k}\sum_{d_1d_2...d_{k}=\delta}\sum_{\substack{t_i|(\prod_{j=i+1}^kd_{j})\\ \gcd(t_i,q/\delta)=1\\ i=1,2,...,k}}\frac{\mu(t_1)\dots \mu(t_k)}{\left( t_1...t_{k}\right)^{s}}\right).
\end{align*}
Thus the main term is
$$M_{k}(x;h,q)=\res_{s=1}\left(\zeta(s)^{k}\frac{x^{s}}{s}f_{k}(q,\delta,s)\right),$$
where, as defined in the statement of the lemma, we have
\begin{align*}
f_{k}(q,\delta,s)&=\frac{1}{\varphi(q/\delta)\delta^s}\left(\sum_{d |(q/\delta)}\frac{\mu(d)}{d^s}\right)^{k}\sum_{d_1d_2...d_{k}=\delta}\sum_{\substack{t_i|(\prod_{j=i+1}^kd_{j})\\ \gcd(t_i,q/\delta)=1\\ i=1,2,...,k}}\frac{\mu(t_1)\dots \mu(t_k)}{\left(t_1...t_{k}\right)^{s}}.
\end{align*}
Smith \cite{Smith} conjectured the validity of the estimate $D_k(0,h,q)\ll q^{\frac{k-1}{2}+\varepsilon}$ for any $(q,h)=1$. This was later affirmed by Matsumoto \cite{MR792769}. Therefore we have the bound
\begin{align*}
\Delta_k(x;h,q)&\ll\sum_{d_1...d_{k}=\delta}\sum_{\substack{t_i|(\prod_{j=i+1}^kd_{j})\\ i=1,2,...,k}}\left|\mu(t_1)\dots \mu(t_k)\right|\left(\left({q}/{\delta}\right)^{\frac{k-1}{2}+\varepsilon}+q^{\varepsilon}x^{\frac{k-1}{k+1}+\varepsilon}\right)\\
&\ll_k \left(q^{\frac{k-1}{2}+\varepsilon}+x^{\frac{k-1}{k+1}+\varepsilon}\right)\sum_{d_1...d_{k}=\delta}\tau(\delta)^{k-1}\ll x^{1-\frac{2}{k+1}+\varepsilon},
\end{align*}
using $q\leq x^{\frac 2{k+1}}$ and $\tau_k(\delta)\ll\delta^\varepsilon$.
This completes the proof of the lemma.
\end{proof}

\begin{lemma}\label{t22} Let $q\ge 1$ be an integer, $(a,q)=1$ and denote $\delta=(h,q)$. Also let $f_k(q,\delta, s)$ be defined as in Lemma \ref{t21}. Define
$$\Phi_{k,a}(q,s)=\sum_{h=1}^q e\left(-\frac{ah}{q}\right)f_k(q,\delta,s).$$
Then $\Phi_{k,a}(q,s)$ is independent of $a$ and we may write it as $\Phi_k(q,s)$.
Furthermore, $\Phi_k(q,s)$ is multiplicative function and
$$\frac{\,d^r \Phi_k(q,1)}{\,ds^r}\ll_{k} q^{-1+\varepsilon}$$
holds for each integer $r=0,1,...,k-1$.
\end{lemma}
\begin{proof}
First, we have
\begin{align}\label{eq:Phi}
\Phi_{k,a}(q,s)&=\sum_{\delta|q}\sum_{\substack{1\le h\le q\\ \gcd(h,q)=\delta}}e\left(-\frac{ah}{q}\right)f_k(q,\delta,s)=\sum_{\delta|q}f_k(q,\delta,s)\sum_{\substack{1\le h_1\le q/\delta\\ \gcd(h_1,q/\delta)=1}}e\left(-\frac{ah_1}{q/\delta}\right)\nonumber\\
&=\sum_{\delta|q}c_{\delta}(a)f_k(q,q/\delta,s)=\sum_{\delta|q}\mu(\delta)f_k(q,q/\delta,s),
\end{align}
where $c_{\delta}(a)$ is the Ramanujan's sum and we use the fact that if $(a,q/\delta)=(a,q)=1$ then $c_{\delta}(a)=\mu(\delta)$. Therefore $F_{k,a}(q,s)$ is independent on $a$. Suppose that the positive integers $q_1$ and $q_2$ are coprime, then
\begin{align*}
\Phi_k(q_1,s)\Phi_k(q_2,s)&=\sum_{\delta_2|q_2}\sum_{\delta_1|q_1}\mu(\delta_1)\mu(\delta_2)f_k(q_1,q_1/\delta_1,s)f_k(q_2,q_2/\delta_2,s)\\
&=\sum_{(\delta_1\delta_2)|(q_1q_2)}\mu(\delta_1\delta_2)f_k(q_1,q_1/\delta_1,s)f_k(q_2,q_2/\delta_2,s),
\end{align*}
hence we just need to show that
$$f_k(q_1,q_1/\delta_1,s)f_k(q_2,q_2/\delta_2,s)=f_k(q_1q_2,q_1q_2/(\delta_1\delta_2),s)$$
whenever $\delta_1|q_1$ and $\delta_2|q_2$. For this we use the definition of $f_k(q,q/\delta,s)$, namely
$$
f_k(q,q/\delta,s)=\frac{\delta^s}{\varphi(\delta)q^s}\left(\sum_{d |\delta}\frac{\mu(d)}{d^s}\right)^{k}\sum_{d_1d_2...d_{k}=q/\delta}\sum_{\substack{t_i|(\prod_{j=i+1}^kd_{j})\\ \gcd(t_i,\delta)=1\\ i=1,2,...,k}}\frac{\mu(t_1)\dots \mu(t_k)}{\left(t_1...t_{k}\right)^{s}}.
$$
For $\sigma={\Re}(s)$ we obtain
$$
f_k(q,q/\delta,s)\ll \frac{\delta^{\sigma}}{\varphi(\delta)q^{\sigma}}\prod_{p|\delta}\left(1+\frac{1}{p^{\sigma}}\right)^k\sum_{d_1d_2...d_{k}=q/\delta}\prod_{i=1}^{k}\prod_{\substack{p|(\prod_{j=i+1}^kd_{j})\\ \gcd(p,\delta)=1}}\left(1+\frac{1}{p^{\sigma}}\right).
$$
Let us assume that $s$ lies on a circle with a centre $s=1$, so we can write $s=1+\rho e(\theta)$ with $\theta\in[0,1)$ and $\rho\in(0,1)$. Then it is easy to see that
$$
f_k(q,q/\delta,s)\ll \frac{\delta^{\sigma}}{\varphi(\delta)q^{\sigma}} 2^{k\omega(\delta)}\tau_k(q)2^{k\omega(q)}\ll q^{\varepsilon} \frac{\delta^{\sigma}}{\varphi(\delta)q^{\sigma}}.
$$
Here $\omega(n)$ is the number of distinct prime factors of $n$ and we used the well known fact that $\omega(n)\ll \frac{\log n}{\log\log n}$ as $n\rightarrow \infty$.  Thus we have
$$\Phi_{k}(q,s)\ll q^{\varepsilon}\sum_{\delta|q}\left|\mu(\delta)\right|\frac{\delta^{\sigma}}{\varphi(\delta)q^{\sigma}}= q^{-\sigma+\varepsilon}\prod_{p|q}\left(1+\frac{p^{\sigma}}{p-1}\right)\ll q^{-\sigma+\varepsilon}\prod_{p|q}\left(1+\frac{p^{\sigma}}{p}\right).$$
On the other hand, when $\sigma\in\left(0,2\right)$, we have
\begin{align*}
q^{-\sigma}\prod_{p|q}\left(1+\frac{p^{\sigma}}{p}\right)\ll \begin{cases}
q^{-\sigma+\varepsilon} \quad &\sigma\in(0,1];\\
q^{-\sigma+\varepsilon}\prod_{p|q}p^{-1+\sigma}\ll q^{-1+\varepsilon} & \sigma\in(1,2).
\end{cases}
\end{align*}
Therefore for $\sigma={\Re}(s)$, $0<\sigma<2$, we get
\begin{equation}\label{pree}
\Phi_k(q,s)\ll q^{-\min(\sigma, 1)+\varepsilon}.
\end{equation}
It is obvious that $\Phi_{k}(q,s)$ is analytic for every $s\in\cb$, and for every parameter $q$ which we consider. Hence one can use Cauchy's integral formula:
\begin{equation*}
\frac{\,d^r\Phi_{k}(q,s)}{\,ds^r}{\bigg|}_{s=1}=\frac{r!}{2\pi i}\int_{|\xi-1|=\rho}\frac{\Phi_{k}(q,\xi)}{(\xi-1)^{r+1}}\,d\xi\ll \frac{r!}{\rho^r}\max_{\theta\in[0, 1)}\left|\Phi_{k}(q,1+\rho e(\theta))\right|,
\end{equation*}
where $\rho\in (0,1)$.
Using (\ref{pree}) and choosing $\rho\ll\varepsilon$, we obtain
$$\frac{\,d^r \Phi_k(q,1)}{\,ds^r}\ll \frac{r!}{\rho^r}q^{-(1-\rho)+\varepsilon}\ll_{k, \varepsilon} q^{-1+\varepsilon},$$
as $q\rightarrow\infty$, which completes the proof of the lemma.
\end{proof}

Now we can deal with the representation of the sum $T(\alpha, Y)$.
\begin{proof}[Proof of Lemma \ref{lemkdf}]

First of all, we pick $Y\asymp X^2$. Recall that by Lemma \ref{t21} for $q\le X^{2/(k+1)}$ and $\beta=\alpha-a/q$ we have
\begin{align*}J_{k}(\alpha, Y)&=\sum_{h=1}^qe\left(\frac{ah}{q}\right)\sum_{\substack{m\le X\\ m\equiv h\pmod q}}\tau_k(m)e(m\beta)\\
&=\sum_{h=1}^qe\left(\frac{ah}{q}\right)\int_{0}^Ye(u\beta)\,d\left(M_k(u;h,q)+O_k(u^{1-\frac{2}{k+1}+\varepsilon})\right)\\
&=\sum_{h=1}^qe\left(\frac{ah}{q}\right)\int_{0}^Ye(u\beta)M'(u;h,q)\,du+O_k\left(q(1+|\beta|Y)Y^{1-\frac{2}{k+1}+\varepsilon}\right).
\end{align*}
Here we also used a summation formula described for example in \cite[Lemma 3.7]{GuoZhai2012}.
It is clear that
$$\sum_{h=1}^qe\left(\frac{ah}{q}\right)M'(u;h,q)=\sum_{h=1}^qe\left(\frac{ah}{q}\right)\res_{s=1}\left(\zeta(s)^{k}u^{s-1}f_{k}(q,\delta,s)\right),$$
where $\delta=(q,h)$. This means that
\begin{equation}\label{eqjk}
T(\alpha, Y)=\int_{0}^{Y}e(u\beta)\res_{s=1}\left(\zeta(s)^k\Phi_k(q,s)u^{s-1}\right)\,du+O\left(q(1+|\beta|Y)Y^{1-\frac{2}{k+1}+\varepsilon}\right).
\end{equation}
We now compute
$
\res_{s=1}\left(\zeta(s)^k\Phi_k(q,s)u^{s-1}\right).
$
The Riemann zeta function has a Laurent series about $s = 1$,
\[
\zeta(s)=\frac{1}{s-1}+\sum_{n=0}^{\infty}\frac{(-1)^{n}\gamma_n}{n!}(s-1)^n,
\]
where
\[\gamma_n=\lim_{M\rightarrow\infty}\left(\sum_{d=1}^{M}\frac{\log^{n}d}{d}-\frac{\log^{n+1}M}{n+1}\right),\;\; n\in\zb_{\ge 0}\]
are the Stieltjes constants. Therefore there exist constants
$$\alpha_{k,j}=\res_{s=1}\left((s-1)^{j-1}\zeta(s)^k\right), j=1,2,\dots, k,$$
and a holomorphic function $h_k(s)$ on $\cb$ such that
\begin{equation*}
\zeta(s)^{k}=\sum_{r=1}^{k}\frac{\alpha_{k,r}}{(s-1)^r}+h_k(s).
\end{equation*}
Thus we obtain that
\begin{equation*}
\zeta(s)^{k}u^{s-1}=\sum_{r=1}^{k}\frac{1}{(s-1)^r}\sum_{r_1=0}^{k-r}\alpha_{k,r_1+r}\frac{\log^{r_1} u}{r_1!}+g_{k,u}(s),
\end{equation*}
for any $u>0$, where $g_{k,u}(s)$ is a holomorphic function on $\cb$ about $s$. The Taylor series for  $\Phi_k(q,s)$ at $s=1$ is
\[\Phi_k(q,s)=\sum_{d=0}^{\infty}\frac{\Phi_k^{\langle d\rangle}(q,1)}{d !}(s-1)^{d}.\]
Therefore the residue of $\zeta(s)^{k}x^{s-1}\Phi_k(q,s)$ at $s=1$ is
\begin{equation*}
\sum_{\substack{r-d=1\\ d, r\in\zb_+,1\le r\le k}}\frac{\Phi_k^{\langle d\rangle}(q,1)}{d !}\sum_{r_1=0}^{k-r}\alpha_{k,r_1+r}\frac{\log^{r_1} x}{r_1!}=\sum_{r=1}^{k}\frac{\log^{r-1}x}{(r-1)!}\sum_{t=0}^{k-r}\Phi_k^{\langle t\rangle}(q,1)\frac{\alpha_{k,r+t}}{t!}.
\end{equation*}
Thus if we define
$$
\beta_{k,r}(q)=\frac{1}{r!}\sum_{t=0}^{k-r-1}\frac{1}{t!}\res_{s=1}\left((s-1)^{r+t}\zeta(s)^k\right)\left(\frac{\,d^t \Phi_k(q,s)}{\,ds^t}\bigg|_{s=1}\right)
$$
by Lemma \ref{t22} we obtain $\beta_{k,r}(q)\ll q^{-1+\varepsilon}$. Furthermore, the error term in \eqref{eqjk} is
$$q(1+|\beta|Y)Y^{1-\frac{2}{k+1}+\varepsilon}\ll q(1+L/q)X^{2-\frac{4}{k+1}+\varepsilon}\ll LX^{2-\frac{4}{k+1}+\varepsilon}$$
for $q\ll L=o\left(X^{\min\left(1,\frac{4}{k+1}\right)}\right)$, which completes the proof of Lemma \ref{lemkdf}.
\end{proof}

\section{The singular integral and series}

\subsection{The singular integral}\label{sec41}~

In this subsection we deal with the singular integral and give a proof of Lemma \ref{lemsi}. We first proof the following lemmas.
\begin{lemma}\label{lem41}
Let $\beta\in\rb\setminus\{0\}$ and $Y\ge 2$. We have
$$\int_{0}^{Y}e\left(-u\beta\right)(\log u)^r\,du\ll |\beta|^{-1+\varepsilon}Y^{\varepsilon}.$$
\end{lemma}
\begin{proof}
We have
\begin{align*}
\int_{0}^{Y}e\left(-u\beta\right)(\log u)^r\,du&\ll |\beta|^{-1}\int_{0}^{Y|\beta|}e\left(-u\beta/|\beta|\right)(\log (u/|\beta|))^r\,du\\
&\ll |\beta|^{-1} \sum_{\ell=0}^r|\log |\beta||^{r-\ell}\left|\int_{0}^{Y|\beta|}(\log u)^{\ell}e\left(-u\frac{\beta}{|\beta|}\right)\,du\right|\\
&\ll |\beta|^{-1}\left(Y^{\varepsilon}|\beta|^{\varepsilon}+1+\sum_{\ell=1}^{r}\int_{1}^{Y|\beta|}\frac{|\log u|^{\ell-1}}{u}\,du\right)\ll |\beta|^{-1+\varepsilon}Y^{\varepsilon}.
\end{align*}
This completes the proof.
\end{proof}
\begin{lemma}\label{lem42}
Let $F({\bf t})$ be defined as in \eqref{def:F} and $X\ge 2$. If $\beta\in\rb$ and $|\beta|\ge X^{-2}$ then
$$I_{F, \mathcal{B}}(\beta, X):=\int_{X\mathcal{B}}e(F({\bf t})\beta)\,d{\bf t}\ll |\beta|^{-n/2+\varepsilon}.$$
\end{lemma}
\begin{proof}
First, we notice that from the fact that $Q$ is nonsingular it follows that there exists a transformation, such that
\begin{align*}
\int_{X\mathcal{B}}e\left(F({\bf t})\beta\right)\,d{\bf t}&=\int_{X\mathcal{B}}e\left(\left({\bf t}^TQ{\bf t}+{\bf L}^T{\bf t}+N\right)\beta\right)\,d{\bf t}\\
&\ll \left|\int_{X\mathcal{B}}e\left(\left({\bf t}^TQ{\bf t}+{\bf L}^T{\bf t}\right)\beta\right)\,d{\bf t}\right|
\ll \left|\int_{X\mathcal{B}+{\bf b_F}}e\left({\bf y}^TQ{\bf y}\beta\right)\,d{\bf y}\right|
\end{align*}
for some  ${\bf b_F}\in \rb^n$. Here $X\mathcal{B}+{\bf b_F}$ is still a box, i.e. a factor of intervals, and we can consider that $\mathcal{B}+{\bf b_F}/X$ has a maximal side length smaller than $1$.  According to  \cite[Lemma 5.2]{Birch} of Birch, for a quadratic nonsingular form $G$ and a box $\mathfrak{C}$ with a maximal side length smaller than $1$, we have
$$I_{G,\mathfrak{C}}(\beta,1)\ll |\beta|^{-n/2+\varepsilon},$$
where the dependence in this version is uniform on the side length of the box $\mathfrak{C}$.  Indeed, we apply \cite[Lemma 5.2]{Birch} with $K=n/2, R=1, d=2$, after we have noticed that the condition (iii) from \cite[Lemma 3.2]{Birch} is not fulfilled for $k=(K-\varepsilon)\Theta$, thus \cite[Lemma 4.3]{Birch} holds in our case too, therefore Lemma 5.2 of Birch applies for our form $Q$. We point out this, since a direct look of the main theorem of Birch implies $n\geq 5$, which is, however, superfluous for \cite[Lemma 5.2]{Birch}. Therefore we have
\begin{align*}\int_{X\mathcal{B}+{\bf b_F}}e\left({\bf y}^TQ{\bf y}\beta\right)\,d{\bf y}=I_{Q,\mathcal{B}+{\bf b_F}/X}(\beta,X)&=X^{-n}I_{Q,\mathcal{B}+{\bf b_F}/X}(\beta X^{-2},1)\ll |\beta|^{-n/2+\varepsilon}X^{-2\varepsilon}\ll |\beta|^{-n/2+\varepsilon}.
\end{align*}
This completes the proof of the lemma.
\end{proof}

\begin{proof}[Proof of Lemma \ref{lemsi}] Using Lemma \ref{lem41} and Lemma \ref{lem42}, we obtain that
\begin{align*}
I_{r,F}(\beta, X):=\int_{X\mathcal{B}}\,d{\bf t}\int_{0}^{C_{F,\mathcal{B}}(X)}&e\left((F({\bf t})-u)\beta\right)(\log u)^r\,du\ll_F |\beta|^{-1-n/2+\varepsilon}X^{\varepsilon}.
\end{align*}
This implies that
\begin{equation}\label{ierror}
\int_{|\beta|\le L/qX^2}I_{r,F}(\beta, X)\,d\beta=\int_{\rb}I_{r,F}(\beta, X)\,d\beta+O\left(X^{\varepsilon}(L/qX^2)^{-\frac{n}{2}}\right).
\end{equation}
Moreover,
\begin{align*}
\int_{\rb}I_{r,F}(\beta, X)\,d\beta&=\int_{\rb}\,d\beta\int_{X\mathcal{B}}\,d{\bf t}\int_{0}^{C_{F,\mathcal{B}}(X)}e\left((F({\bf t})-u)\beta\right)(\log u)^r\,du\\
&=2\int_{\rb_+}\,d\beta\int_{0}^{C_{F,\mathcal{B}}(X)}(\log u)^r\,du\int_{X\mathcal{B}}\cos\left[2\pi(u-F({\bf t}))\beta\right]\,d{\bf t}\\
&=\frac{1}{\pi}\int_{X\mathcal{B}}\,d{\bf t}\int_{\rb_+}\,d\beta\int_{0}^{C_{F,\mathcal{B}}(X)}(\log u)^r\, d\left(\frac{\sin\left[2\pi(u-F({\bf t}))\beta\right]}{\beta}\right)\\
&=\frac{1}{\pi}\int_{X\mathcal{B}}\,d{\bf t}\int_{0}^{C_{F,\mathcal{B}}(X)}(\log u)^r\,d\left(\int_{\rb_+}\frac{\sin\left[2\pi(u-F({\bf t}))\beta\right]}{\beta}\,d\beta\right)\\
&=\frac{1}{\pi}\int_{X\mathcal{B}}\,d{\bf t}\int_{0}^{C_{F,\mathcal{B}}(X)}(\log u)^r\,d\left(\frac{\pi}{2}{\rm sgn}(u-F({\bf t}))\right),
\end{align*}
where we have used the fact: $\int_{0}^{\infty}\frac{\sin(\alpha x)}{x}\mathrm{d}x=\frac{\pi}{2}{\rm sgn}(\alpha)$ and
\begin{align*}
{\rm sgn}(\alpha):=\begin{cases}\frac{\alpha}{\left|\alpha\right|} \qquad &\alpha\neq 0\\
~~0\qquad &\alpha=0.
\end{cases}
\end{align*}
By integration by parts we have
\begin{align*}
\rint_{\rb}I_{r,F}(\beta, X)\,d\beta&=\frac{1}{2}\int_{X\mathcal{B}}\,d{\bf t}\int_{0}^{C_{F,\mathcal{B}}(X)}(\log u)^r\,d\left({\rm sgn}(u-F({\bf t}))\right)\\
&=\lim_{\epsilon\rightarrow 0^+}\int_{X\mathcal{B}}\frac{\,d{\bf t}}{2}\int_{\substack{|u-F({\bf t})|\le \epsilon\\ 0\le u\le C_{F,\mathcal{B}}(X)}}(\log u)^r\,d\left({\rm sgn}(u-F({\bf t}))\right)\\
&=\lim_{\epsilon\rightarrow 0^+}\rint_{X\mathcal{B}}\frac{\,d{\bf t}}{2}\left(\left.(\log u)^r\left({\rm sgn}(u-F({\bf t}))\right)\right|_{F({\bf t})-\varepsilon}^{F({\bf t})+\varepsilon}-\rint_{\substack{|u-F({\bf t})|\le \epsilon\\ 0\le u\le C_{F,\mathcal{B}}(X)}}{\rm sgn}(u-F({\bf t}))\,d(\log u)^r\right)\\
&=\int_{X\mathcal{B}}\frac{1}{2}\left(2(\log F({\bf t}))^r\,d{\bf t} +\lim_{\epsilon\rightarrow 0^+} O(\epsilon \log^rX)\right)\,d{\bf t}\\
&=\int_{X\mathcal{B}}(\log F({\bf t}))^r\,d{\bf t}.
\end{align*}
Using (\ref{ierror}) we get the proof of Lemma \ref{lemsi}. \end{proof}

\subsection{The singular series}\label{lemss}~

In this subsection we deal with the singular series, i.e. with the series $L(s;k,F)$ defined in \eqref{def:L}, and their presentation stated in Theorem \ref{mth}. \newline

First of all, note that from Lemma \ref{lemqf} it follows that $S_F(q)\ll q^{1-n/2+\varepsilon}$ and Lemma \ref{t22} gives $\displaystyle\frac{d^r\Phi_k(q,1)}{ds^r}\ll q^{-1+\varepsilon} $ for any integer $r\in [0,k-1]$. Hence, for any $t=0,\ldots,k-1$,
$$\left.\frac{d^t L(s; k,F)}{ds^t}\right|_{s=1}=\sum_{q=1}^\infty \frac{d^t\Phi_k(q,1)}{ds^t}S_F(q)\ll\sum_{q=1}^\infty q^{-n/2+\varepsilon}\ll 1\,,$$
as $n\geq 3$. By their definition in Theorem \ref{mth} this ensures that $C_{k,r}(F)$, $r=0,\ldots,k-1$, are convergent and indeed well-defined constants.

It is easily seen that $S_F(q)$ defined in \eqref{def:S_F}
is real and multiplicative. On the other hand, Lemma \ref{t22} showed that $\Phi_k(q,s)$ is also multiplicative. Therefore $L(s; k,F)=\sum_{q=1}^\infty \Phi_k(q,s)S_F(q)$ has  an Euler product representation as follows:
\begin{equation*}
L(s; k,F)=\prod_{p}L_p(s;k,F)
\end{equation*}
with
$$L_p(s;k,F)=1+\sum_{m\ge 1}S_F(p^{m})\Phi_k(p^{m},s)$$
\newline
By orthogonality of characters in $\zb/p^m\zb$ for integer $m\ge 1$ it easily follows that
 \[\varrho_F(p^m)=p^{-nm}\sum_{1\leq a\leq p^m}S_F(p^m,a).\]
Then we have
\begin{equation}\label{Srho}
S_F(p^m)=\varrho_F(p^m)-\varrho_F(p^{m-1}).
\end{equation}
By the estimate from Lemma \ref{lemqf} we get $S_F(p^m)\ll_F (p^m)^{1-n/2+\varepsilon}$ and after telescoping summation of \eqref{Srho} we obtain
$$\varrho_F(p^\ell)-1\ll_F \sum_{m=1}^\ell (p^m)^{1-n/2+\varepsilon}\ll p^{1-n/2+\varepsilon},$$
where we again used that $n\geq 3$. Then by partial summation, using  \eqref{Srho} and the estimate \eqref{pree}, we have
\[L_p(s;k,F)=\sum_{\ell\ge 0}\varrho_F(p^{\ell})\left(\Phi_k(p^{\ell},s)-\Phi_k(p^{\ell+1},s)\right),\]
where we set $\varrho_F(1)=\Phi_k(1,s)=1$. \newline


From \eqref{eq:Phi} and the definition of $f_k(q,\delta,s)$ in Lemma \ref{t21}, we see that
\begin{align*}
\Phi_{k}(p^m,s)&=f_k(p^m,p^m,s)-f_k(p^m,p^{m-1},s)\\
&=\frac{1}{p^{ms}}\sum_{d_1d_2...d_{k}=p^m}\sum_{\substack{t_i|(\prod_{j=i+1}^kd_{j})\\ i=1,2,...,k}}\frac{\mu(t_1)\dots \mu(t_k)}{\left(t_1...t_{k}\right)^{s}}-\frac{\left(1-p^{-s}\right)^{k}}{\varphi(p)p^{(m-1)s}}\tau_k(p^{m-1})\\
\end{align*}
For the first expression above, denote
\[
I_k=\sum_{d_1d_2...d_{k}=p^m}\sum_{\substack{t_i|(\prod_{j=i+1}^kd_{j})\\ i=1,2,...,k}}\frac{\mu(t_1)\dots \mu(t_k)}{\left(t_1...t_{k}\right)^{s}}.\]
Then for $m\ge 1$ and $k=2$ we have
\[I_2=1+m(1-p^{-s}).\]
Now using the identities $\tau_k(p^m)=\sum_{v=0}^m \tau_{k-1}(p^{m-v})$, from which it also follows that
\begin{equation}\label{tau_kminus}\tau_k(p^m)-\tau_{k-1}(p^m)=\tau_k(p^{m-1}),
\end{equation}
 we see that
\begin{align*}
I_k&=\sum_{v=0}^{m}\sum_{d_1d_2...d_{k-1}=p^{m-v}}\sum_{\substack{t_i|(\prod_{j=i+1}^kd_{j})\\ i=1,2,...,k}}\frac{\mu(t_1)\dots \mu(t_k)}{\left(t_1...t_{k}\right)^{s}}\\
&=\sum_{d_1d_2...d_{k-1}=p^{m}}\sum_{\substack{t_i|(\prod_{j=i+1}^{k-1}d_{j})\\ i=1,2,...,k-1}}\frac{\mu(t_1)\dots \mu(t_{k-1})}{\left(t_1...t_{k-1}\right)^{s}}+\sum_{v=1}^{m}\sum_{d_1d_2...d_{k-1}=p^{m-v}}\left(1-\frac{1}{p^s}\right)^{k-1}\\
&=I_{k-1}+\left(1-p^{-s}\right)^{k-1}\sum_{v=1}^{m}\tau_{k-1}(p^{m-v})=I_{k-1}+\left(1-p^{-s}\right)^{k-1}\left(\tau_k(p^m)-\tau_{k-1}(p^m)\right)\\
&=1+m(1-p^{-s})+\sum_{v=3}^{k}\left(1-p^{-s}\right)^{v-1}\tau_{v}(p^{m-1})=\sum_{v=1}^{k}\left(1-p^{-s}\right)^{v-1}\tau_{v}(p^{m-1}).
\end{align*}
Hence
\begin{equation}\label{Phi_v2}
\Phi_k(p^m,s)=p^{-ms}\left(\sum_{1\le v\le k}(1-p^{-s})^{v-1}\tau_{v}(p^{m-1})-\tau_k(p^{m-1})\frac{p^s(1-p^{-s})^{k}}{p-1}\right).
\end{equation}\newline
We now aim to find the value of $\Phi_k(p^m,s)-\Phi_k(p^{m+1},s)$ for each non-negative integer $m$. When $m=1$ we have
\begin{align*}
\Phi_k(1,s)-\Phi_k(p,s)&=1-p^{-s}\left(\sum_{v=1}^{k}(1-p^{-s})^{v-1}\tau_{v}(1)-\tau_k(1)\frac{p^s(1-p^{-s})^{k}}{p-1}\right)\\
&=1-p^{-s}\left(\frac{1-(1-p^{-s})^k}{1-(1-p^{-s})}-\frac{p^s(1-p^{-s})^k}{p-1}\right)\\
&=(1-p^{-s})^k\frac{p}{p-1}=(1-p^{-1})^{-1}(1-p^{-s})^k.
\end{align*}
If $f(z)$ is a formal power series, we denote by $[z^n]f(z)$ the coefficient of $z^n$ in $f(z)$. Then for any $|z|<1$ and $m, v\in\zb_+$ we have
$$\tau_v(p^{m-1})=[z^{m-1}]\left((1-z)^{-v}\right).$$
Since the symbol $[z^n]f(z)$ has a distributive property, we have
\begin{align*}
\phi_k(p^m,s)&:=\frac{1}{p^{ms}}\sum_{v=1}^{k}(1-p^{-s})^{v-1}\tau_v(p^{m-1})\\
&=[z^{m-1}]\left(\frac{1}{p^{ms}}\sum_{v=1}^{k}\frac{(1-p^{-s})^{v-1}}{(1-z)^v}\right)\\
&=[z^{m-1}]\left(\frac{1}{p^{(m-1)s}}\frac{1}{1-p^sz}\left(1-\frac{(1-p^{-s})^{k}}{(1-z)^k}\right)\right)\\
&=1-p^{(1-m)s}(1-p^{-s})^{k}[z^{m-1}]\left((1-p^sz)^{-1}(1-z)^{-k}\right)\\
&=1-(1-p^{-s})^{k}\sum_{0\le \ell\le m-1}p^{-s\ell }[z^{\ell}](1-z)^{-k}\\
&=1-(1-p^{-s})^{k}\left((1-p^{-s})^{-k}-\sum_{\ell\ge m}p^{-s\ell}[z^{\ell}](1-z)^{-k}\right)\\
&=(1-p^{-s})^{k}\sum_{\ell\ge m}p^{-s\ell}[z^{\ell}](1-z)^{-k}=(1-p^{-s})^{k}\sum_{\ell\ge m}p^{-s\ell}\tau_k(p^\ell).
\end{align*}
and then for $m\ge 1$ we get
$$\phi_k(p^m,s)-\phi_k(p^{m+1},s)=(1-p^{-s})^k p^{-ms}\tau_k(p^m).$$
From \eqref{Phi_v2} it follows that when $m\geq 1$ we have
\begin{align*}
\Phi_k(p^m,s)-\Phi_k(p^{m+1},s)=&\left(\phi_k(p^m,s)-\frac{(1-p^{-s})^{k}p^s}{p^{sm}(p-1)}\tau_k(p^{m-1})\right)\\
&-\left(\phi_k(p^{m+1},s)-\frac{(1-p^{-s})^{k}p^s}{p^{s(m+1)}(p-1)}\tau_k(p^{m})\right)\\
=&(1-p^{-s})^kp^{-ms}\left(\tau_{k}(p^m)-\frac{p^s}{p-1}\left(\tau_{k}(p^{m-1})-\frac{\tau_k(p^m)}{p^s}\right)\right)\\
=&\frac{(1-p^{-s})^k}{1-p^{-1}}p^{-ms}\left(\tau_{k}(p^m)-p^{s-1}\tau_{k}(p^{m-1})\right).
\end{align*}
Let $\sigma:=\Re(s)>0$. Then according to \eqref{pree} we have $\Phi_k(p^\ell,s)\rightarrow 0$, as $\ell\rightarrow 0$ and $s$ is fixed. Then after appropriate telescoping summation we can write
\begin{align*}
L_p(s;k,F)&=1+\sum_{\ell\ge 0}(\varrho_F(p^{\ell})-1)(\Phi_k(p^\ell,s)-\Phi_k(p^{\ell+1},s))\\
&=1+\sum_{\ell\ge 1}O\left(p^{1-n/2+\varepsilon}p^{-\ell\sigma}\left(\tau_{k}(p^{\ell})+p^{\sigma-1}\tau_{k}(p^{\ell-1})\right)\right)\end{align*}
Let us further assume that $\sigma>1/2$, so that we obtain
$$L_p(s;k,F)\ll 1+O\left(p^{1-n/2+\varepsilon-\sigma}(1+p^{\sigma-1})\right)=1+O(p^{-n/2+\varepsilon}+p^{1-n/2-\sigma+\varepsilon}).$$
Therefore if $\sigma>\max(1/2,2-n/2)=1/2$, and setting $\tau_{k}(p^{-1}):=0$, we have that the Euler product
$$L(s;k,F)=\prod_{p}\left(\sum_{\ell\ge 0}\frac{\varrho_F(p^{\ell})\left(\tau_{k}(p^{\ell})-p^{s-1}\tau_{k}(p^{\ell-1})\right)}{p^{\ell s}}\right)\left(\frac{(1-p^{-s})^k}{1-p^{-1}}\right)\,,$$
is absolutely convergent. In particular, by \eqref{tau_kminus} we have
$$L(1;k,F)=\prod_{p}\left(\sum_{\ell\ge 0}\frac{\varrho_F(p^{\ell})\tau_{k-1}(p^{\ell})}{p^{\ell}}\right)\left(1-\frac{1}{p}\right)^{k-1}>0\,.$$
Now from $$C_{k,k-1}=\frac{1}{(k-1)!}L(1;k,F)\res_{s=1}\left((s-1)^{k-1}\zeta(s)^k\right)=\frac{L(1;k,F)}{(k-1)!}$$ we conclude that $C_{k,k-1}>0$, which finalizes the proof of Theorem \ref{mth}.


\section{Final remarks}
We believe that the application of the circle method in estimating divisor sums over values of quadratic polynomials can be extended also to the sum
\[\Sigma^\ell_{k,F}(X; {\mathcal{B}}):=\sum_{{\bf x}\in X\mathcal{B}\cap\mathbb{Z}^{d}}\tau^\ell_{k}\left(F({\bf x})\right)\,.\]
The treatment of the sum $S(\alpha)$ remains the same, and one could use a level of distribution result for the function $\tau^\ell_k(m)$ given by Rieger (Satz 3, \cite{Rieger}). In this case a separate treatment for $q$ in the middle range $(\log x)^\lambda\le q\le L\ll X$ might also be required.

\bigskip
\noindent
{\sc Institute of Analysis and Number Theory\\
Graz University of Technology \\
Kopernikusgasse 24/II\\
8010 Graz\\
Austria}\newline
\href{mailto:lapkova@math.tugraz.at}{\small lapkova@math.tugraz.at}

\bigskip
\noindent
{\sc School of Mathematical Sciences\\
East China Normal University\\
500 Dongchuan Road\\
Shanghai 200241\\
PR China}\newline
\href{mailto:nianhongzhou@outlook.com}{\small nianhongzhou@outlook.com}

\end{document}